\documentclass[a4paper,10pt]{amsart}

\usepackage[utf8]{inputenc} 
\usepackage{amsmath}
\usepackage{amsbsy}
\usepackage{amssymb}
\usepackage{amscd,amsthm}
\usepackage{graphicx}
\usepackage[mathcal]{eucal}
\usepackage{verbatim}
\usepackage[english]{babel}
\usepackage[dvigs]{epsfig}
\usepackage{dsfont}
\usepackage{longtable}
\usepackage{marvosym}
\usepackage{anysize}
\usepackage{hyperref}

\newcommand{\e}{\varepsilon}

\newcommand{\Z}{\mathds{Z}}

\newcommand{\Pz}{\mathds{P}}

\newcommand{\q}{\quad}

\newtheorem{thm}{Theorem}[section]
\newtheorem{lem}[thm]{Lemma}
\newtheorem{kor}[thm]{Corollary}

\newtheorem{prop}[thm]{Proposition}

\theoremstyle{definition}
\newtheorem*{defi}{Definition}

\theoremstyle{remark}
\newtheorem*{rema}{Remark}

\title{Inequalities involving the primorial counting function}
\author{Christian Axler}
\address{Institute of Mathematics\\ Heinrich-Heine University Duesseldorf\\
40225 Duesseldorf, Germany}
\email{christian.axler@hhu.de}
\date{\today}

\begin{document}

\begin{abstract}
Let $\varphi(n)$ denote the Euler totient function. In this paper, we first establish a new upper bound for $n/\varphi(n)$
involving $K(n)$, the function that counts the number of primorials not exceeding $n$. In particular, this leads to an answer 
to a question raised by Aoudjit, Berkane, and Dusart concerning an upper bound for the sum-of-divisors function $\sigma(n)$. 
Furthermore, we give some lower bounds for $N_k/\varphi(N_k)$ as well as for $\sigma(N_k)/N_k$, where $N_k$ denotes the $k$th 
primorial.
\end{abstract}

\maketitle
\section{Introduction}

Let $\varphi(n)$ denote Euler's totient function, the function which represents the number of positive integers up to $n$
that are relatively prime to $n$. Since $\varphi$ is multiplicative and fulfills $\varphi(p^k) = p^k(1-1/p)$ for any prime 
$p$ and any positive integer $k$, we get
\begin{equation}
\varphi(n) = n\prod_{p \vert n}\left( 1 - \frac{1}{p} \right) \tag{1.1} \label{1.1}
\end{equation}
for every positive integer $n$. Let $p_k$ denote as usual the $k$th prime number with $p_1 = 2$. For every positive integer,
we introduce $N_k$ to be the $k$th primorial; i.e. the product of the first $k$ prime numbers. If we now consider the 
function $f(n) = n/\varphi(n)$, the primorials play an important role in that they are so-called $f$-champions, i.e. for 
every integer $n$ with $1 \leq n < N_{k+1}$, the identity \eqref{1.1} provides that
\begin{equation}
\frac{n}{\varphi(n)} \leq \frac{N_k}{\varphi(N_k)} = \prod_{p \leq p_k} \frac{p}{p-1}. \tag{1.2} \label{1.2}
\end{equation}
So, in order to find effective upper bounds for $n/\varphi(n)$, it is convenient to study the product in \eqref{1.2}.
If we combine the equation in \eqref{1.2} with a result of Vinogradov \cite{vinogradov}, we get
\begin{equation}
\frac{N_k}{\varphi(N_k)} = e^{\gamma} \log p_k + O ( \exp(-a_0\log^{3/5}k)) \tag{1.3} \label{1.3}
\end{equation}
as $k \to \infty$, where $\gamma = 0.5772156\ldots$ denotes the Euler-Mascheroni constant and $a_0$ is a positive absolute
constant. Applying the asymptotic formula for $\log p_k$ found by Cipolla \cite{cp}, we get that
\begin{equation}
\frac{N_k}{\varphi(N_k)} = e^{\gamma} \left( \log k + \log_2 k + \frac{\log_2 k-1}{\log k} - \frac{\log_2^2k - 4\log_2k +
5}{2\log^2k} \right) + O \left( \frac{\log_2^3k}{\log^3k} \right) \tag{1.4} \label{1.4}
\end{equation}
as $k \to\infty$.
Here, and throughout this paper, $\log_j$ denotes the $j$-fold iterated logarithm (for more terms, see Cipolla \cite{cp} or
Arias de Reyna and Toulisse \cite{adrt}). In our first result, we first find an upper bound for $N_k/\varphi(N_k)$ which 
corresponds to the first terms of the asymptotic formula \eqref{1.4} and then we utilize the inequality given in \eqref{1.2} 
to get the following upper bound for $n/\varphi(n)$ involving the primorial counting function $K(x)$, the function that 
counts the number of primorials not exceeding $x$.

\begin{thm} \label{thm101}
For every integer $n \geq N_{740,322}$, we have
\begin{displaymath}
\frac{n}{\varphi(n)} < e^{\gamma} \left( \log K(n) + \log_2 K(n) + \frac{\log_2 K(n)-1}{\log K(n)} - \frac{\log_2^2K(n) -
4\log_2K(n) + 4.897}{2\log^2K(n)}\right).
\end{displaymath}
\end{thm}

\begin{rema}
Note that the inequality given in Theorem \ref{thm101} does not hold for $n= N_{740,321}$.
\end{rema}

On the other hand, we give the following lower bound for $N_k/\varphi(N_k)$.

\begin{thm} \label{thm102}
For every integer $k \geq \exp(\exp(5.879)) = 1.734\ldots \times 10^{155}$, we have
\begin{equation}
\frac{N_k}{\varphi(N_k)} > e^{\gamma} \left( \log k + \log_2 k + \frac{\log_2 k -1}{\log k} - \frac{\log_2^2k - 4\log_2k
+ 5}{2\log^2k}\right). \tag{1.5} \label{1.5}
\end{equation}
\end{thm}

The asymptotic formula \eqref{1.3} can also be used to show that
\begin{displaymath}
\frac{N_k}{\varphi(N_k)} = e^{\gamma} \log_2 N_k + O ( \exp(-a_1\log^{3/5}k))
\end{displaymath}
as $k \to \infty$, where $a_1$ is a positive absolute constant. In \cite[Th\'eor\`eme 2]{nico83}, Nicolas proved that the
Riemann hypothesis is equivalent to
\begin{equation}
\frac{N_k}{\varphi(N_k)} > e^{\gamma} \log_2 N_k \tag{1.6} \label{1.6}
\end{equation}
for every positive integer $k$. We now use this necessary and sufficient criterion to give the following result concerning
the inequality \eqref{1.5}.

\begin{thm} \label{thm103}
If the Riemann hypothesis is true, then the inequality \eqref{1.5} holds for every integer $k \geq 3.900491 \times 10^{30}$.
\end{thm}

Next, we consider the arithmetical function $\sigma$ which is defined by
\begin{displaymath}
\sigma(n) = \sum_{d \vert n} d
\end{displaymath}
and denotes the sum of the divisors of $n$. Gronwall \cite[p. 119]{gronwall1913} found the maximal order of $\sigma$ by
showing that
\begin{displaymath}
\limsup_{n \to \infty} \frac{\sigma(n)}{n \log_2 n} = e^{\gamma}.
\end{displaymath}
Under the assumption that the Riemann hypothesis is true, Ramanujan \cite{ramanujan1997} showed that $\sigma(n)/n <
e^{\gamma} \log_2 n$
for all sufficiently large positive integers $n$. Robin \cite[Th\'eor\`eme 1]{robin1984} improved Ramanujan's result by
showing that the Riemann hypothesis is true \textit{if and only if}
\begin{equation}
\frac{\sigma(n)}{n} < e^{\gamma} \log_2 n \q \q (n > 5040). \tag{1.7} \label{1.7}
\end{equation}
This criterion on the Riemann hypothesis is called \textit{Robin's criterion} and the inequality \eqref{1.7} is called
\textit{Robin's inequality}. Robin's inequality is proved to hold in many cases (see, for instance, Banks et al. 
\cite{banks2009}, Briggs \cite{briggs2006}, Grytczuk \cite{grytczuk2007}, Grytczuk \cite{grytczuk2010}, and Choie et al. 
\cite{choie2007}), but remains open in general. Lagarias \cite[Theorem 1.1]{lagarias} used harmonic numbers to demonstrate an 
alternate form of Robin's assertion \eqref{1.7}, which requires no exceptions. However, the ratio $\sigma(n)/n$ cannot be too 
large. For instance, the present author \cite[Theorem 1.3]{axlerrobin} proved that the inequality 
\begin{displaymath}
\frac{\sigma(n)}{n} < e^{\gamma} \log_2 n + \frac{0.0094243e^{\gamma}}{\log_2^2n}
\end{displaymath}
holds unconditionally for every integer $n > 5040$.
The initial motivation to write this paper is based on a recent paper by Aoudjit, Berkane, and Dusart \cite{berkane} 
concerning another upper bound for $\sigma(n)/n$. Balazard \cite[p.\:257]{balazard} showed that for every positive integer 
$N$, one has
\begin{displaymath}
K(x) = \text{li}(\log x) + O \left( \frac{\log x}{\log_2^{N+1}x} \right) \tag{1.8} \label{1.8}
\end{displaymath}
as $x \to \infty$, where the integral logarithm $\text{li}(x)$ is defined for every $x \geq 0$ as
\begin{displaymath}
\text{li}(x) = \int_0^x \frac{\text{d}t}{\log t} = \lim_{\varepsilon \to 0+} \left \{ 
\int_{0}^{1-\varepsilon}{\frac{\text{d}t}{\log t}} + \int_{1+\varepsilon}^{x}{\frac{\text{d}t}{\log t}} \right \}.
\end{displaymath}
Since the asymptotic formula \eqref{1.8} provides that $\log K(x) \sim \log_2 x$ as $x \to \infty$, Aoudjit, Berkane, and 
Dusart \cite[Theorem 1.2]{berkane} studied in their common paper another form of an upper bound for the quotient 
$\sigma(n)/n$ than that exposed in \eqref{1.7}. They showed that
\begin{displaymath}
\frac{\sigma(n)}{n} \leq e^{\gamma} \left( \log K(n) + \log_2 K(n) + \frac{\log_2 K(n)}{\log K(n)} + \frac{1}{20(\log_2 
K(n))^2}\right)
\end{displaymath}
holds unconditionally for every integer $n \geq 30$. Further, they \cite[Proposition 1]{berkane} assumed to have verified
that the slight better inequality 
\begin{equation}
\frac{\sigma(n)}{n} \leq e^{\gamma} \left( \log K(n) + \log_2 K(n) + \frac{\log_2 K(n)}{\log K(n)} \right) \tag{1.9}
\label{1.9}
\end{equation}
holds unconditionally for every integer $n$ satisfying $205 \leq n \leq CA_{160}$ (however, the inequality \eqref{1.9} does
not hold for $n = 1,680$). Here $CA_k$ denotes the $k$th colossally abundant number (see Section 5) with $CA_{160} = 1.8772 
\ldots \times 10^{326}$. Based on this numerical computation, they \cite[Question 1]{berkane} asked whether the inequality 
\eqref{1.9} holds for every integer $n \geq 205$ and conjectured that the inequality \eqref{1.9} provides a sufficient and 
necessary criterion for the Riemann hypothesis. In this context they \cite[Proposition 2]{berkane} assumed to have proved 
that under the assumption of the Riemann hypothesis the inequality \eqref{1.9} holds for every integer $n \geq 205$. 
Unfortunately, there are some errors in their proof, so that the proof of does not longer work. Nevertheless, we utilize 
Theorem \ref{thm101} to get the following result where we show that an even sharper inequality holds unconditionally. Here, 
and throughout this paper, we let
\begin{equation}
J_0 = 2^9 \times 3^5 \times 5^3 \times 7^2 \times 11^2 \times 13^2 \times 151 \prod_{17\le p \le 139} p = 7.847144\ldots
\times 10^{65}. \tag{1.10} \label{1.10}
\end{equation}
Note that
\begin{equation}
N_{38} < J_0 < N_{39}. \tag{1.11} \label{1.11}
\end{equation}

\begin{thm} \label{thm104}
For every integer $n > J_0$, we have
\begin{equation}
\frac{\sigma(n)}{n} \leq e^{\gamma} \left( \log K(n) + \log_2 K(n) + \frac{\log_2 K(n) - 1}{\log K(n)} - \frac{\log_2^2 K(n)
- 4\log_2 K(n) + 4.897}{2\log^2 K(n)}\right). \tag{1.12} \label{1.12}
\end{equation}
\end{thm}

\begin{rema}
The inequality \eqref{1.12} does \textit{not} hold for $n = J_0$. In other words, $n = J_0$ is the largest positive integer
for which the inequality \eqref{1.12} does not hold. 
\end{rema}

In the proof of Theorem \ref{thm104} the so-called \textit{superabundant numbers} (cf.\:Section 5) play an important role.
Next, we find the following weaker but more compact version of Theorem \ref{thm104}.

\begin{kor} \label{kor105}
For every integer $n > 521,585,633,051,683,200$, we have
\begin{equation}
\frac{\sigma(n)}{n} \leq e^{\gamma} \left( \log K(n) + \log_2 K(n) + \frac{\log_2 K(n)-1}{\log K(n)} \right). \tag{1.13}
\label{1.13}
\end{equation}
\end{kor}

\begin{rema}
The inequality \eqref{1.13} does also not hold for $n= 521,585,633,051,683,200$. 
\end{rema}

Corollary \ref{kor105} now provides an answer to the above question raised by Aoudjit, Berkane, and Dusart 
\cite[Question 1]{berkane} whether the inequality \eqref{1.9} holds unconditionally for every integer $n \geq 205$.

\begin{kor} \label{kor106}
The inequality \eqref{1.9} holds for every integer $n > 1,680$.
\end{kor}

\begin{rema}
The complete list of positive integers for which the inequality \eqref{1.9} does not hold is given by
\begin{align*}
&6,7,8,9,10,11,12,13,14,15,16,17,18,19,20,21,22,23,24,25,26,27,28,29,30,36,42,48,60, \\
&72,80,84,90,96,108,120,126,132,140,144,150,156,160,168,180,192,198,200,204,1680.
\end{align*}
Note that the right-hand side of \eqref{1.9} is not defined for all integers $n$ with $1 \leq n \leq 5$.
\end{rema}

\section{Proof of Theorem \ref{thm101}}

Starting point of the proof of Theorem \ref{thm101} is the inequality given in \eqref{1.2}. Furthermore, we need the 
following both lemmata. In the first one, we give an explicit upper bound for $1/\log^2 p_k$.

\begin{lem} \label{lem201}
For every integer $k \geq 2$, we have
\begin{displaymath}
\frac{1}{\log^2p_k} \leq \frac{1}{\log^2 k} - \frac{2 \log_2 k}{\log^3k} + \frac{3\log_2^2k - 2\log_2k + 2}{\log^4k}.
\end{displaymath}
\end{lem}

\begin{proof}
First, we consider the case where $k \geq 209$. By \cite[Corollary 9]{axlernpr}, we have
\begin{equation}
\frac{1}{\log p_k} \geq \frac{1}{\log k} - \frac{\log_2k}{\log^2k} + \frac{\log_2^2k - \log_2k + 1}{\log^2k\log p_k}.
\tag{2.1} \label{2.1}
\end{equation}
A simple calculation shows that the inequality given in \cite[Corollary 12]{axlernpr} implies that
\begin{equation}
\frac{1}{\log p_k} \leq \frac{1}{\log k} - \frac{\log_2k}{\log^2 k} + \frac{(\log_2k)^2 - \log_2k + 1}{\log^3 k} + 
\frac{P(\log_2k)}{2\log^4 k}, \tag{2.2} \label{2.2}
\end{equation}
where $P(x) = 3x^2-6x+5.2$, The last inequality yields that
\begin{displaymath}
\frac{1}{\log^2p_k} \leq \frac{1}{\log k \log p_k} - \frac{\log_2k}{\log^2 k \log p_k} + \frac{(\log_2k)^2 - \log_2k + 
1}{\log^4 k} + \frac{P(\log_2k)}{2\log^4 k \log p_k}.
\end{displaymath}
If we apply \eqref{2.2} to the term $1/\log k \log p_k$ and \eqref{2.1} to the term $- \log_2k/\log^2 k \log p_k$, we obtain
\begin{displaymath}
\frac{1}{\log^2p_k} \leq \frac{1}{\log^2k} - \frac{2\log_2k}{\log^3k} + \frac{3\log_2^2k-2\log_2k+2}{\log^4k} - 
\frac{\log_2^3k - \log_2^2k + \log_2k - P(\log_2k)}{\log^4k\log p_k}.
\end{displaymath}
It suffices to note that $\log_2^3k - \log_2^2k + \log_2k - P(\log_2k) > 0$ to get the required inequality for every 
$k \geq 209$. Finally, we check the required inequality for every integer $k$ with $2 \leq k < 209$ with a computer.
\end{proof}

In the second lemma, we state an effective upper bound for the product given in \eqref{1.2}.

\begin{lem} \label{lem202}
For every integer $k \geq \pi(10^{19}) = 234,057,667,276,344,607$, we have
\begin{displaymath}
\prod_{p \leq p_k} \frac{p}{p-1} < e^{\gamma} \left( \log p_k + \frac{0.0088067}{\log^2p_k} \right).
\end{displaymath}
\end{lem}

\begin{proof}
Let $k_0$ be given by $k_0 = \pi(10^{19})$. We can utilize Walisch's \textit{primecount} C++ code \cite{walisch} to get 
$p_{k_0} 
= 9,999,999,999,999,999,961$.
An easy calculation shows that
\begin{displaymath}
\exp \left( \frac{0.024334}{3\log^3x} \left( 1  + \frac{15}{4\log x} \right) \right) \leq 1 + \frac{0.0088067}{\log^3x}
\end{displaymath}
for every $x \geq p_{k_0}$ and it sufficient to substitute the last inequality into \cite[Proposition 6.1]{axlernew}.
\end{proof}

Now we use \eqref{1.2} and the Lemmata \ref{lem201} and \ref{lem202} to get the following proof of Theorem \ref{thm101}.

\begin{proof}[Proof of Theorem \ref{thm101}]
First, we note the following preliminary remark. If the inequality
\begin{displaymath}
\frac{N_k}{\varphi(N_k)} < e^{\gamma} \left( \log k + \log_2k + \frac{\log_2k - 1}{\log k} - \frac{\log_2^2k - 4\log_2k +
4.897}{2\log^2k} \right) \tag{2.3} \label{2.3}
\end{displaymath}
holds for some positive integer $k$, then \eqref{1.2} provides that the required inequality holds for every integer $n$ with
$N_k \leq n < N_{k+1}$ (note that in this case, we have $K(n) = k$). So it suffices to show that the inequality \eqref{2.3} 
holds for every integer $k \geq 740,322$.
First, let $k$ be an integer with $k \geq k_0$ where $k_0 = \pi(10^{19})$. If we combine Lemma \ref{lem202} with
\cite[Theorem 1.4]{axlerNthp}, we get that
\begin{displaymath}
 \prod_{p \leq p_k} \frac{p}{p-1} < e^{\gamma} \left( \log k + \log_2k + \log \left( 1 + g_1(k) \right) + 
 \frac{0.0088067}{\log^2p_k} \right),
\end{displaymath}
where
\begin{displaymath}
g_1(k) = \frac{\log_2k-1}{\log k} + \frac{\log_2k-2}{\log^2k} - \frac{\log_2^2k - 6\log_2k + 10.912351}{2\log^3k}.
\end{displaymath}
Now we apply Lemma \ref{lem201} to see that 
\begin{equation}
\prod_{p \leq p_k} \frac{p}{p-1} < e^{\gamma} \left( \log k + \log_2k + \log \left( 1 + g_1(k) \right) + 0.0088067 \times
h(k) \right), \tag{2.4} \label{2.4}
\end{equation}
where
\begin{displaymath}
h(k) = \frac{1}{\log^2k} - \frac{2\log_2k}{\log^3k} + \frac{3\log_2^2 k -2\log_2k + 2}{\log^4k}.
\end{displaymath}
If we substitute the inequality
\begin{displaymath}
\log \left( 1 + g_1(x) \right) + 0.0088067 \times h(x) \leq \frac{\log_2x - 1}{\log x} - \frac{\log_2^2x - 4\log_2x +
4.897}{2\log^2x},
\end{displaymath}
which holds for every $x \geq 8$, into \eqref{2.4}, we get that the inequality \eqref{2.3} holds for every $k \geq k_0$.
Next, we verify the inequality \eqref{2.3} for every integer $k$ with $564,397,542 \leq k < k_0$. In \cite{axlernicolas}, it
is shown that the inequality 
\begin{equation}
\frac{n}{\varphi(n)} \leq  e^{\gamma} \left( \log \log n + \frac{0.0094243}{(\log \log n)^2} \right) \tag{2.5} \label{2.5}
\end{equation}
holds for every integer $n \geq N_{564,397,542}$. Chebyshev's $\vartheta$-function is defined by
\begin{equation}
\vartheta(x) = \sum_{p \leq x} \log p, \tag{2.6} \label{2.6}
\end{equation}
where $p$ runs over primes not exceeding $x$. Since $\log N_k = \vartheta(p_k)$, we can utilize \eqref{2.5} and the 
inequality $\vartheta(p_m) \geq m$, which holds for every integer $m \geq 3$, to see that
\begin{displaymath}
\frac{N_k}{\varphi(N_k)} \leq e^{\gamma} \left( \log \vartheta(p_k) + \frac{0.0094243}{\log^2k} \right)
\end{displaymath}
for every integer $k \geq 564,397,542$.
Now we can use \cite[Theorem 1.8]{axlerNthp} to get
\begin{equation}
\frac{N_k}{\varphi(N_k)} \leq e^{\gamma} \left( \log k + \log_2 k + \log(1+g_2(k)) + \frac{0.0094243}{\log^2k} \right), 
\tag{2.7} \label{2.7}
\end{equation}
where
\begin{displaymath}
g_2(k) = \frac{\log_2k-1}{\log k} + \frac{\log_2k-2}{\log^2k} - \frac{\log_2^2k - 6\log_2k + 11}{2\log^3k}.
\end{displaymath}
Note that
\begin{equation}
\log (1 + g_2(x)) + \frac{0.0094243}{\log^2x} \leq \frac{\log_2x - 1}{\log x} - \frac{\log_2^2x - 4\log_2x + 
4.94488}{2\log^2x} \tag{2.8} \label{2.8}
\end{equation}
for every $x$ with $e < x \leq k_0$. If we substitute the inequality \eqref{2.8} into \eqref{2.7}, we see that the inequality \eqref{2.3} holds for every integer $k$ satisfying $564,397,542 \leq k < k_0$. To complete the proof, we use the right-hand side identity of \eqref{1.2} and a computer to check the inequality \eqref{2.3} for every integer $k$ with $740,322 \leq k \leq 564,397,542$.
\end{proof}

\section{Proof of Theorem \ref{thm102}}

For every $x > 1$ let
\begin{displaymath}
A_1(x) = \sum_{p \leq x} \frac{1}{p} - \log_2 x - B,
\end{displaymath}
where $B$ denotes the \textit{Mertens' constant}
and is defined by
\begin{displaymath}
B = \gamma + \sum_{p} \left( \log \left( 1 - \frac{1}{p} \right) + \frac{1}{p} \right) = 0.26149 \ldots.
\end{displaymath}
Mertens \cite[p.\:52]{mertens1874} found that $A_1(x) = O(1/\log x)$ as $x \to \infty$. Rosser and Schoenfeld 
\cite[p.\:74]{rosserschoenfeld1962} derived a remarkable identity which connects $A_1(x)$ with Chebyshev's 
$\vartheta$-function by showing that
\begin{equation}
A_1(x) =  \frac{\vartheta(x) - x}{x \log x} - \int_x^{\infty} \frac{(\vartheta(y)-y)(1 + \log y)}{y^2\log^2y} \, \text{d}y.
\tag{3.1} \label{3.1}
\end{equation}
In Lemma \ref{lem202}, we found an upper bound for the product
\begin{equation}
\prod_{p \leq p_k} \frac{p}{p-1} \tag{3.2} \label{3.2}
\end{equation}
In order to prove Theorem \ref{thm102}, we first note the following lemma where we utilize the identity \eqref{3.1} to give
a lower bound for the product in \eqref{3.2}.

\begin{lem} \label{lem601}
For every integer $k \geq \pi(e^{337}) + 1$, one has
\begin{displaymath}
\prod_{p \leq p_k} \frac{p}{p-1} > e^{\gamma} \left( \log p_k - \frac{14.5}{\log^3p_k}\right).
\end{displaymath}
\end{lem}

\begin{proof}
If we apply \cite[Lemma 3.1]{axlernew} to \eqref{3.1}, it turns out that
\begin{displaymath}
A_1(p_k) \geq - \frac{57.184}{4\log^4p_k} \left( 1 + \frac{24}{5\log p_k} \right).
\end{displaymath}
Now we can substitute this inequality into \cite[Equation (6.4)]{axlernew} to see that
\begin{displaymath}
\prod_{p \leq p_k} \frac{p}{p-1} > e^{\gamma} \log p_k \exp\left( - \frac{57.184}{4\log^4p_k} \left( 1 + \frac{24}{5\log
p_k} \right) + S(p_k) \right),
\end{displaymath}
where $S(x) = - \sum_{n=2}^{\infty} \sum_{p > x} 1/(np^n)$. By Rosser and Schoenfeld \cite[p.\:87]{rosserschoenfeld1962},
we have $S(x) > - 1.02/((x-1)\log x)$ for every $x > 1$. Hence,
\begin{displaymath}
\prod_{p \leq p_k} \frac{p}{p-1} > e^{\gamma} \log p_k \exp\left( - \frac{57.184}{4\log^4p_k} \left( 1 + \frac{24}{5\log 
p_k} \right) - \frac{1.02}{(p_k-1)\log p_k} \right).
\end{displaymath}
Now it is again an easy exercise to show that
\begin{displaymath}
\exp\left( - \frac{57.184}{4\log^4t} \left( 1 + \frac{24}{5\log t} \right) - \frac{1.02}{(t-1)\log t} \right) > 1 - 
\frac{14.5}{\log^4t}
\end{displaymath}
for every $t \geq e^{337}$ and we arrive at the end of the proof.
\end{proof}

Next, we give a proof of Theorem \ref{thm102}, where we find a lower bound for $N_k/\varphi(N_k)$ which corresponds to the
first terms of the asymptotic formula \eqref{1.4}.

\begin{proof}[Proof of Theorem \ref{thm102}]
Let $k$ be an integer with $k \geq \exp(\exp(5.879))$. If we combine Lemma \ref{lem601} with \cite[Theorem 1.1]{axlerNthp},
we obtain that
\begin{displaymath}
e^{-\gamma} \prod_{p \leq p_k} \frac{p}{p-1} > \log k + \log_2k + \log \left( 1 + g_3(k) \right) - \frac{14.5}{\log^3k}, 
\tag{3.3} \label{3.3}
\end{displaymath}
where
\begin{displaymath}
g_3(k) = \frac{\log_2k-1}{\log k} + \frac{\log_2k-2}{\log^2k} - \frac{\log_2^2k - 6\log_2k + 11.0991617}{2\log^3k}.
\end{displaymath}
We have
\begin{displaymath}
\log \left( 1 + g_3(x) \right) - \frac{14.5}{\log^3x} \geq \frac{\log_2x - 1}{\log x} - \frac{\log_2^2x - 4\log_2x + 
5}{2\log^2x}
\end{displaymath}
for every $x \geq \exp(\exp(5.879))$. To complete the proof, it suffices to substitute the last inequality into \eqref{3.3}.
%
\end{proof}

\section{Proof of Theorem \ref{thm103}}

The proof of Theorem \ref{thm103} is comparatively simple.

\begin{proof}[Proof of Theorem \ref{thm103}]
We assume that the Riemann hypothesis is true. If we combine \eqref{1.6} with the identity $\log N_k = \vartheta(p_k)$, we
obtain that the inequality $N_k/\varphi(N_k) > e^{\gamma}\log \vartheta(p_k)$ holds for every positive integer $k$. Applying
\cite[Theorem 1.7]{axlerNthp} to see that
\begin{displaymath}
\frac{N_k}{\varphi(N_k)} > e^{\gamma} \left( \log k + \log_2k + \log \left(1 + \frac{\log_2k - 1}{\log k} + \frac{\log_2k 
- 2}{\log^2k} - \frac{\log_2^2k - 6\log_2k + 11}{2\log^3k} \right) \right)
\end{displaymath}
for every integer $k \geq 3.900491 \times 10^{30}$. Now we can argue as in the proof of Theorem \ref{thm102} to get that the
required inequality \eqref{1.5} is fulfilled for every integer $k \geq 3.900491 \times 10^{30}$.
\end{proof}

\section{Proof of Theorem \ref{thm104}}

In order to give a proof of Theorem \ref{thm104}, we first give some necessary preliminaries. The sum-of-divisors function
$\sigma$ is multiplicative and satisfies for any prime number $p$ and any nonnegative integer $m$ the identity
\begin{equation}
\sigma(p^m) = 1 + p + \cdots + p^{m-1} + p^m = \frac{p^{m+1}-1}{p-1}. \tag{5.1} \label{5.1}
\end{equation}
If $n$ is a positive integer with factorization $n = q_1^{e_1} \cdot \ldots \cdot q_k^{e_k}$, where $q_i$ are primes and 
$e_i \geq 1$, we cancombine \eqref{5.1} and \eqref{1.1} to see that
the sum-of-divisors function $\sigma$ and Euler's totient function $\varphi$ are connected by the identity
\begin{displaymath}
\frac{\sigma(n)}{n} = \frac{n}{\varphi(n)} \prod_{i=1}^k \left( 1 - \frac{1}{q_i^{1 + e_i}} \right), \tag{5.2} \label{5.2}
\end{displaymath}
which gives
\begin{equation}
\frac{\sigma(n)}{n} < \frac{n}{\varphi(n)} \tag{5.3} \label{5.3}
\end{equation}
for every integer $n \geq 2$. The inequality \eqref{5.3} and, as already mentioned in the introduction, the so-called
\textit{superabundant numbers} play also an important role in the proof of Theorem \ref{thm104}.

\begin{defi}
A positive integer $N$ is said to be \textit{superabundant} (or a \textit{SA number}) if
\begin{displaymath}
\frac{\sigma(n)}{n} < \frac{\sigma(N)}{N} \q\q (\text{$0 < n < N$}).
\end{displaymath}
\end{defi}

Superabundant numbers were first introduced by Ramanujan \cite{ramanujan1997}, who called them generalized highly composite.
They also have been introduced and studied by Alaoglu and Erdös \cite{erdos1944}. The first superabundant numbers are (see 
OEIS \cite{sloane})
\begin{displaymath}
N = 1, 2, 4, 6, 12, 24, 36, 48, 60, 120, 180, 240, 360, 720, 840, 1260, 1680,\ldots.
\end{displaymath}
We will use the symbol $SA_k$ to denote the $k$th superabundant number. If $n$ is a SA number with $n= \prod_{p \in \Pz} 
p^{\nu_p(n)}$, where $\Pz$ denotes the set of all prime numners and $\nu_p(n)$ denotes the $p$-adic valuation of an integer 
$n$, Alaoglu and Erdös \cite[Theorem 1]{erdos1944} found $\nu_p(n)$ is non-increasing in $p$. If $p$ denotes the largest 
prime factor of the SA number $n$, they could also show that $\nu_p(n) = 1$, except for $n=4$ and $n=36$. These properties 
were used by Noe \cite{noe} to compute all SA numbers up to $SA_{1,000,000}$. Superabundant numbers also play an important 
role in the verification of Robin's inequality \eqref{1.7}. Akbary and Friggstad \cite[Theorem 3]{akbary2009}, proved that if 
there is any counterexample to Robin’s inequality \eqref{1.7}, then the least such counterexample is a superabundant number. 
We also need the following numbers.

\begin{defi}
A positive integer $N$ is said to be \textit{colossally abundant} (or a \textit{CA number}) if there exists a real number
$\e > 0$ such that
\begin{equation}
\frac{\sigma(n)}{n^{1+ \e}} \leq \frac{\sigma(N)}{N^{1+ \e}} \tag{5.4} \label{5.4}
\end{equation}
for every positive integers $n$. The number $\e$ is called a \textit{parameter} of the CA number $N$.
\end{defi}

Below, we recall some properties of colossally abundant numbers. For more details, see, for instance, Alaoglu and Erdös
\cite{erdos1944}, Ramanujan \cite{ramanujan1997}, Erdös and Nicolas \cite{ernic}, Caveney, Nicolas, and Sondow 
\cite{caveney2012}, and Nazardonyavi and Yakubovich \cite{nazardonyavi2014}. First, it is obvious to see that every CA number 
is also a SA number. Further, for every $\e > 0$ there exists at least one CA number of parameter $\e$ (see Erdös and Nicolas 
\cite[Proposition 2]{ernic}). For $\e > 0$, let us introduce the number
\begin{displaymath}
M_{\e} = \prod_{p \in \Pz} p^{\lfloor \mu \rfloor}
\end{displaymath}
where
\begin{displaymath}
\mu = \mu(p,\e) = \frac{\log((p^{1+\e}-1)/(p^{1+\e}-p))}{\log p}.
\end{displaymath}

\begin{lem} \label{lem501}
The number $M_{\e}$ is a CA number of parameter $\e$.
\end{lem}

\begin{proof}
See Alaoglu and Erdös \cite[p.\:455]{erdos1944}. Alternatively, we can apply the method used in the proof of 
\cite[Proposition 3.7]{nicolas3}.
\end{proof}

In the proof of Theorem \ref{thm104}, we first show that the required inequality \eqref{1.12} holds for every integer
$n \geq N_{39}$, where
\begin{displaymath}
N_{39} = 9.629\ldots \times 10^{65}.
\end{displaymath}
A simple calculation shows that the inequality \eqref{1.12} does not hold for $n = SA_{425}$, where
\begin{displaymath}
SA_{425} = 2^9 \times 3^5 \times 5^3 \times 7^2 \times 11^2 \times 13^2 \times \prod_{17\le p \le 149} p = 7.743\ldots
\times 10^{65},
\end{displaymath}
So it remains to consider the interval $(SA_{425}, N_{39})$. In order to show that $J_0$ (cf. \eqref{1.10}) is the largest
positive integer in the interval $(SA_{425}, N_{39})$ for which the inequality \eqref{1.12} does not hold, we need the 
following notation.

%
%


\begin{defi}
Let $\e > 0$ and let $N$ be a CA number of parameter $\e$. For a positive integer $n$, we introduce the \textit{benefit}
of $n$ by
\begin{equation}
\text{ben}_{\e}(n) = \log \left( \frac{\sigma(N)}{N^{1+ \e}} \right) - \log \left( \frac{\sigma(n)}{n^{1+ \e}} \right) = \log \left( \frac{\sigma(N)}{\sigma(n)} \right) + (1+ \e)(\log n - \log N). \tag{5.5} \label{5.5}
\end{equation}
\end{defi}

\begin{rema}
The function $\text{ben}_{\e}(n)$ is well-defined. Indeed, if $N'$ is another CA number of parameter $\e > 0$, then the
value of the right-hand side of \eqref{5.5} does not change when we replace $N$ by $N'$. 
\end{rema}

If $N$ be a CA number of parameter $\e$ then \eqref{5.4} implies that $\text{ben}_{\e}(n) \geq 0$ for every positive integer
$n$. In \cite{nicolas2}, Nicolas was able to show the following result.

\begin{lem}[Nicolas] \label{lem502}
Let $\e > 0$ and let $N$ be a CA number of parameter $\e$. If $\beta$ is a positive real number, then the set of all 
positive integers $k$ satisfying $\emph{ben}_{\e}(k) \leq \beta$ is finite.
\end{lem}

\begin{proof}
See Nicolas \cite[Proposition 4.14]{nicolas2}.
\end{proof}

Also, Nicolas \cite{nicolas2} investigated an algorithm to compute all integers $k$ such that $\text{ben}_{\e}(k) \leq 
\beta$ for given $\e > 0$ and $\beta > 0$. This algorithm is efficient if $\beta$ is not too large (not much larger than 
$\e$). Now we use this algorithm to show the following result which will be very helpful for the proof of Theorem 
\ref{thm104}.
 

\begin{prop} \label{prop503}
The only integer $n$ in the interval $(SA_{425}, N_{39})$ satisfying
\begin{equation}
\frac{\sigma(n)}{n} > 8.8272 \tag{5.6} \label{5.6}
\end{equation}
is given by $n = J_0$.
\end{prop}

\begin{proof}
Let $I = (SA_{425}, N_{39})$ and let $n \in I$ be an integer satisfying \eqref{5.6}.
Further let $\e = 0.00133$.
By Lemma \ref{lem501}, we get that $M_{\e} = 
SA_{425}$ is a CA number of parameter $\e$. Now we can combine the definition \eqref{5.5} of the benefit of $n$ with 
\eqref{5.6} to see that
\begin{displaymath}
\text{ben}_{\e}(n) \leq \log(\sigma(SA_{425})) - (1+\e) \log SA_{425} + \e \log n - \log 8.8272.
\end{displaymath}
Since $n \leq N_{39}-1$, it turns out that $\text{ben}_{\e}(n) \leq \beta$, where $\beta = 0.000594$. On the other hand, we 
can utilize Nicolas' algorithm which is based on the symbolic algebra system Maple to compute the set of all positive 
integers $k \in I$ satisfying $\text{ben}_{\e}(k) \leq \beta$. Let $H = \{k \in \Z \cap I : \text{ben}_{\e}(k) \leq \beta\}$. 
It suffices to write, in Maple,
\begin{displaymath}
\texttt{bensmall($SA_{425}, \e, \beta, SA_{425}+1, N_{39}-1$, 4);}
\end{displaymath}
to see that $|H| = 2$. If we then use the command
\begin{displaymath}
\texttt{seq(lisbenshort[k][1], k=1..nops(lisbenshort));},
\end{displaymath}
we get the both elements $h_1$ and $h_2$ of $H$, namely
\begin{displaymath}
h_1 = \frac{151}{149} \times SA_{425} = J_0 \q \text{and} \q h_2 = \frac{157}{149} \times SA_{425}.
\end{displaymath}
Finally, a simple calculation shows that $J_0$ fulfills the inequality \eqref{5.6} while $h_2$ violates this inequality.
\end{proof}

Let $M_1$ are $M_2$ consecutive CA numbers satisfying the inequality \eqref{1.7}. Then, Robin 
\cite[Proposition 1]{robin1984} showed that Robin's inequality \eqref{1.7} holds for every integer $n$ such that $M_1 \leq n 
\leq M_2$. Briggs \cite{briggs2006} described an algorithm that computes all CA numbers not exceeding $10^{10^{10}}$ and 
verifies Robin’s inequality for each of these numbers.
Hence, Robin's inequality is fulfilled for every integer $n$ so that $5040 < n \leq 10^{10^{10}}$. Recently, Morrill and
Platt \cite{morrillplatt} extended Briggs result to every integer $n$ with $5040 < n \leq 10^{10^{13.11485}}$. The argument 
in the proof of the last expansion implies a slightly better result.

\begin{lem}[Morrill and Platt] \label{lem504}
Let $k_0 = 999,999,476,056$. Then $p_{k_0} = 29,996,208,012,611$ and Robin's inequality \eqref{1.7} holds for every integer
$n$ so that $5040 < n \leq N_{k_0}$.
\end{lem}

In the last lemma, that we need in the proof of Theorem \ref{thm104}, we note an upper bound for Chebyshev's 
$\vartheta$-function (cf.\:\eqref{2.6}) obtained by Büthe \cite[Theorem 2]{buethe}.

\begin{lem}[B\"uthe] \label{lem505}
For every $x$ satisfying $0 < x \leq 10^{19}$, one has $\vartheta(x) < x$.
\end{lem}

Now we utilize \eqref{5.3}, Theorem \ref{thm101}, Proposition \ref{prop503}, Lemma \ref{lem504}, and Lemma \ref{lem505} to
obtain the following proof of Theorem \ref{thm104}.

\begin{proof}[Proof of Theorem \ref{thm104}]
First, we can combine the inequality given in Theorem \ref{thm101} with \eqref{5.3} to see that the required inequality
\eqref{1.12} holds for every integer $n \geq N_{740,322}$. Next, we show that the inequality \eqref{1.12} holds for every 
integer $n$ satisfying $N_{515} \leq n < N_{740,322}$. For this purpose, we set
\begin{displaymath}
F(x) = \log x + \log_2x + \frac{\log_2x-1}{\log x} - \frac{\log_2^2x - 4\log_2x + 4.897}{2\log^2x}.
\end{displaymath}
We check with a computer that the inequality
\begin{equation}
\log_2 N_{k+1} < F(k) \tag{5.10} \label{5.10}
\end{equation}
holds for every integer $k$ so that $515 \leq k \leq 740,322$. Let $n$ be an integer with $N_{515} \leq n < N_{740,322}$ 
and let $k$ be the unique positive integer so that $N_k \leq n < N_{k+1}$. Then $515 \leq k \leq 740,321$ and $K(n) = k$, and 
\eqref{5.10} implies $e^{\gamma} \log_2n < e^{\gamma} F(K(n))$.
Now we can use Lemma \ref{lem504} to see that the required inequality \eqref{1.12} also holds for every integer $n$ with
$N_{515} \leq n < N_{740,322}$.
Next, we prove the required inequality \eqref{1.12} for every integer $n$ with $N_{60} \leq n < N_{515}$. According to
Nicolas \cite[p.\:27]{nicolas2}, we have
\begin{equation}
\frac{\sigma(m)}{m} < e^{\gamma} \left( \log_2 m - \frac{0.603}{\sqrt{\log m}} \right) \tag{5.11} \label{5.11}
\end{equation}
for every integer $m \in J$, where $J = [SA_{123}, \exp(10^9)]$. Note that $[N_{60}, N_{515}) \subset J$. Then, we check 
with a computer that
\begin{equation}
\log_2 N_{k+1} - \frac{0.603}{\sqrt{\log N_{k+1}}} < F(k) \tag{5.12} \label{5.12}
\end{equation}
for every integer $k$ with $60 \leq k \leq 514$. If $n$ is an integer with $N_{60} \leq n < N_{515}$, we let $k$ be
again the unique positive integer so that $N_k \leq n < N_{k+1}$. Then \eqref{5.11} and \eqref{5.12} imply the required 
inequality \eqref{1.12} for every integer $n$ with $N_{60} \leq n < N_{515}$. To prove the theorem for every integer $n$ with 
$N_{25} \leq n < N_{60}$, let $a$ and $b$ be positive integers with $a < b$. Note that $F(x)$ is an increasing function for 
every $x > 1$. Define $m = m(b)$ to be the smallest positive integer with $N_b \leq SA_m$ and let $n$ be a positive integer 
with $N_a \leq n < N_b$. By the definition of superabundant numbers, we have $\sigma(n)/n < \sigma(SA_m)/SA_m$. If 
$\sigma(SA_m)/SA_m < e^{\gamma}F(a)$, we conclude that the required inequality holds for every positive integer $n$ with $N_a 
\leq n < N_b$. Now we can use this observation combined with the table involving superabundant numbers computed by Noe 
\cite{noe} to get the following table:
\begin{center}
\begin{tabular}{|c|c|c|c|c|c|c|}
\hline
$a$\rule{0mm}{4mm}   &  $b$  & $\log_{10}N_b$      &      $m$ & $\log_{10}SA_m$     & $\sigma(SA_m)/SA_m$ & $e^{\gamma}F(a)$            \\ \hline \hline
$59$\rule{0mm}{4mm}  & $60$  & $115.391780\ldots$  & $784$    & $115.415202\ldots$  & $9.849479\ldots$    & $9.875818\ldots$ \\ \hline
$58$\rule{0mm}{4mm}  & $59$  & $112.943073\ldots$  & $766$    & $113.015528\ldots$  & $9.810393\ldots$    & $9.875818\ldots$ \\ \hline
$57$\rule{0mm}{4mm}  & $58$  & $110.500593\ldots$  & $749$    & $110.676077\ldots$  & $9.771724\ldots$    & $9.795336\ldots$ \\ \hline
$56$\rule{0mm}{4mm}  & $57$  & $108.067624\ldots$  & $733$    & $108.095692\ldots$  & $9.730790\ldots$    & $9.753928\ldots$ \\ \hline
$55$\rule{0mm}{4mm}  & $56$  & $105.637872\ldots$  & $717$    & $105.717294\ldots$  & $9.690245\ldots$    & $9.711702\ldots$ \\ \hline
$54$\rule{0mm}{4mm}  & $55$  & $103.217916\ldots$  & $700$    & $103.247276\ldots$  & $9.644670\ldots$    & $9.668625\ldots$ \\ \hline
$53$\rule{0mm}{4mm}  & $54$  & $100.807983\ldots$  & $684$    & $100.887440\ldots$  & $9.602737\ldots$    & $9.624664\ldots$ \\ \hline
$52$\rule{0mm}{4mm}  & $53$  & $ 98.408309\ldots$  & $669$    & $ 98.531414\ldots$  & $9.560620\ldots$    & $9.579781\ldots$ \\ \hline
$51$\rule{0mm}{4mm}  & $52$  & $ 96.026292\ldots$  & $652$    & $ 96.183110\ldots$  & $9.517938\ldots$    & $9.533938\ldots$ \\ \hline
$50$\rule{0mm}{4mm}  & $51$  & $ 93.647894\ldots$  & $633$    & $ 93.660459\ldots$  & $9.467867\ldots$    & $9.487093\ldots$ \\ \hline
$49$\rule{0mm}{4mm}  & $50$  & $ 91.280538\ldots$  & $618$    & $ 91.361606\ldots$  & $9.420528\ldots$    & $9.439203\ldots$ \\ \hline
$48$\rule{0mm}{4mm}  & $49$  & $ 88.920703\ldots$  & $601$    & $ 89.067140\ldots$  & $9.372950\ldots$    & $9.390220\ldots$ \\ \hline
$47$\rule{0mm}{4mm}  & $48$  & $ 86.564677\ldots$  & $579$    & $ 86.582010\ldots$  & $9.322341\ldots$    & $9.340094\ldots$ \\ \hline
$46$\rule{0mm}{4mm}  & $47$  & $ 84.216372\ldots$  & $558$    & $ 84.221796\ldots$  & $9.272648\ldots$    & $9.288772\ldots$ \\ \hline
$45$\rule{0mm}{4mm}  & $46$  & $ 81.892090\ldots$  & $540$    & $ 81.964117\ldots$  & $9.221700\ldots$    & $9.236195\ldots$ \\ \hline
$44$\rule{0mm}{4mm}  & $45$  & $ 79.593237\ldots$  & $522$    & $ 79.711264\ldots$  & $9.170468\ldots$    & $9.182303\ldots$ \\ \hline
$43$\rule{0mm}{4mm}  & $44$  & $ 77.298770\ldots$  & $505$    & $ 77.473218\ldots$  & $9.117764\ldots$    & $9.127029\ldots$ \\ \hline
$42$\rule{0mm}{4mm}  & $43$  & $ 75.013213\ldots$  & $490$    & $ 75.028653\ldots$  & $9.060176\ldots$    & $9.070301\ldots$ \\ \hline
$41$\rule{0mm}{4mm}  & $42$  & $ 72.732180\ldots$  & $472$    & $ 72.777015\ldots$  & $9.004245\ldots$    & $9.012042\ldots$ \\ \hline
$40$\rule{0mm}{4mm}  & $41$  & $ 70.474501\ldots$  & $456$    & $ 70.564827\ldots$  & $8.949341\ldots$    & $8.952168\ldots$ \\ \hline
\end{tabular}
\end{center}
This table shows that the required inequality \eqref{1.12} holds for every integer $n$ with $N_{40} \leq n < N_{60}$. Next, we verify the required inequality \eqref{1.12} for every integer $n$ with $N_{39} \leq n < N_{40}$. We have $\max\{ m \in \Z_{\geq 0} \mid SA_m < N_{40} \} = 440$. Hence,
\begin{displaymath}
\frac{\sigma(n)}{n} \leq \frac{\sigma(SA_{440})}{SA_{440}} = 8.888355\ldots < 8.890590\ldots = e^{\gamma} F(39) = e^{\gamma} F(K(n)),
\end{displaymath}
as desired. Finally, we need to show that the required inequality \eqref{1.12} also holds for every integer $n \in (J_0, N_{39})$ where $J_0$ is defined by \eqref{1.10}. Since $(J_0, N_{39}) \subset (N_{38}, N_{39})$, we have $K(n) = 38$ for every $n \in (J_0, N_{39})$. According to Proposition \ref{prop503}, we have
\begin{displaymath}
\frac{\sigma(n)}{n} \leq 8.8272 < 8.827208\ldots = e^{\gamma} F(38) = e^{\gamma} F(K(n))
\end{displaymath}
for every integer $n \in (J_0, N_{39})$, which completes the proof.
\end{proof}

Let $k$ be a positive integer. In the case where $n=N_k$ is a primorial, we can use \eqref{1.2} and \eqref{1.3} to see that
\begin{displaymath}
\frac{\sigma(N_k)}{N_k} = \frac{e^{\gamma}}{\zeta(2)} \, \log p_k  + O( e^{-a(\log k)^{3/5}})
\end{displaymath}
as $k \to \infty$, where $a$ is a positive absolute constant. Applying the well known identity $\zeta(2) = \pi^2/6$ and the asymptotic formula for $\log p_k$ found by Cipolla \cite{cp}, we get that
\begin{displaymath}
\frac{\sigma(N_k)}{N_k} = \frac{6e^{\gamma}}{\pi^2} \left( \log k + \log_2 k + \frac{\log_2 k-1}{\log k} - \frac{\log_2^2k - 4\log_2k + 5}{2\log^2k} \right) + O \left( \frac{\log_2^4k}{\log^4k} \right) 
\end{displaymath}
as $k \to\infty$. Now we can utilize Theorems \ref{thm101}-\ref{thm103} to get the following results for $\sigma(N_k)/N_k$.

\begin{prop} \label{prop505}
For every $k \geq 734,170$, one has
\begin{equation}
\frac{\sigma(N_k)}{N_k} < \frac{6e^{\gamma}}{\pi^2} \left( 1 + \frac{1}{p_k} \right) \left( \log k + \log_2 k + \frac{\log_2 k-1}{\log k} - \frac{\log_2^2k - 4\log_2k + 4.897}{2\log^2k}\right) \tag{5.13} \label{5.13}
\end{equation}
and for every $k \geq \exp(\exp(5.879))$, one has
\begin{equation}
\frac{\sigma(N_k)}{N_k} > \frac{6e^{\gamma}}{\pi^2} \left( \log k + \log_2 k + \frac{\log_2 k -1}{\log k} - \frac{\log_2^2k - 4\log_2k + 5}{2\log^2k}\right). \tag{5.14} \label{5.14}
\end{equation}
If the Riemann hypothesis is true, then the inequality \eqref{5.14} holds for every $k \geq 3.900491 \times 10^{30}$.
\end{prop}

\begin{proof}
If we combine \eqref{5.2} with \cite[Lemma 4.3]{dusart20182}, we get
\begin{equation}
\frac{e^{\gamma}}{\zeta(2)} \times \frac{N_k}{\varphi(N_k)} \leq \frac{\sigma(N_k)}{N_k} \leq \frac{e^{\gamma}}{\zeta(2)} \left( 1 + \frac{1}{p_k} \right) \frac{N_k}{\varphi(N_k)} \tag{5.15} \label{5.15}
\end{equation}
Now we can apply Theorem \ref{thm101} to the right-hand side inequality of \eqref{5.15} and get the required upper bound \eqref{5.13} for every integer $k \geq 740,322$. A direct computer check provides that the required inequality \eqref{5.13} also holds for every integer $k$ with $734,170 \leq k < 740,322$. On the other hand, it suffices to apply Theorem \ref{thm102} (resp. Theorem \ref{thm103}) to the left-hand side inequality of \eqref{5.15} to see that the inequality \eqref{5.14} holds for every $k \geq \exp(\exp(5.879))$ (resp. for every $k \geq 3.900491 \times 10^{30}$).
\end{proof}

\section{Proof of Corollary \ref{kor105}}

We start with the following both preliminary results. In the first one, one has
\begin{align*}
SA_{106} &= 224,403,121,196,654,400, \\
SA_{107} &= 448,806,242,393,308,800, \\
N_{15} &= 614,889,782,588,491,410.
\end{align*}

\begin{prop} \label{prop601}
The only two integers $n$ in the interval $(SA_{107}, N_{15})$ satisfying
\begin{equation}
\frac{\sigma(n)}{n} > \frac{\sigma(SA_{106})}{SA_{106}} \tag{6.1} \label{6.1}
\end{equation}
are $n_1 = 497,325,836,165,558,400$ and $n_2 = 521,585,633,051,683,200$.
\end{prop}

\begin{proof}
First, we easily check with a computer that $\sigma(n_1)/n_1 > \sigma(SA_{106})/SA_{106}$ and that $\sigma(n_2)/n_2 > \sigma(SA_{106})/SA_{106}$. Next, let $n$ be an integer with $SA_{107} < n < N_{15}$ satisfying \eqref{6.1}. We need to show that $n = n_1$ or $n = n_2$. In order to do this, we set $I = (SA_{107}, N_{15})$ and $\e = 0.0065$
By Lemma \ref{lem501}, we get that $M_{\e} = 224,403,121,196,654,400 = SA_{106}$ is a CA number of parameter $\e$. Now we can use the definition \eqref{5.5} of the benefit of $n$ to see that
\begin{displaymath}
\text{ben}_{\e}(n) = \log \left( \frac{\sigma(SA_{106})}{\sigma(n)} \right) + (1+ \e)(\log n - \log SA_{106}).
\end{displaymath}
Applying \eqref{6.1}, we get that the inequality
\begin{displaymath}
\text{ben}_{\e}(n) \leq \e(\log n - \log SA_{106}).
\end{displaymath}
Since $n \leq N_{15}-1$, it turns out that $\text{ben}_{\e}(n) \leq \beta$, where $\beta = 0.007$. Similar to the proof of Proposition \ref{prop503}, it turns out that
\begin{displaymath}
\{k \in \Z \cap I : \text{ben}_{\e}(k) \leq \beta\} = \{ n_1, n_2 \}.
\end{displaymath}
Hence $n=n_1$ or $n= n_2$ and we arrive at the end of the proof.
\end{proof}

%
%

Next, we proceed as in the proof of Theorem \ref{thm104} to give a proof of Corollary \ref{kor105}.

\begin{proof}[Proof of Corollary \ref{kor105}]
According to \eqref{1.11} and Theorem \ref{thm104}, it remains to consider the case where $n$ is an integer satisfying $521,585,633,051,683,200 < n < N_{39}$. Let $a$ and $b$ be positive integers with $a < b$ and let
\begin{displaymath}
G(x) = e^{\gamma} \left( \log x + \log_2 x + \frac{\log_2 x-1}{\log x} \right).
\end{displaymath}
Note that $G$ is an increasing function for every $x > 1$. As in the proof of Theorem \ref{thm104}, we define $m = m(b)$ to be the smallest positive integer with $N_b \leq SA_m$ and let $n$ be a positive integer with $N_a \leq n < N_b$. Using the calculations of Noe \cite{noe}, we get the following table:
\begin{center}
\begin{tabular}{|c|c|c|c|c|c|c|}
\hline
$a$\rule{0mm}{4mm}  &  $b$ & $\log_{10}N_b$     &   $m$ & $\log_{10}SA_m$    & $\sigma(SA_m)/SA_m$ & $G(a)$ \\ \hline \hline
$37$\rule{0mm}{4mm} & $39$ & $65.983602\ldots$ & $426$ & $66.189950\ldots$ & $8.834195\ldots$ & $8.858198\ldots$ \\ \hline
$35$\rule{0mm}{4mm} & $37$ & $61.548698\ldots$ & $395$ & $61.572719\ldots$ & $8.708363\ldots$ & $8.726022\ldots$ \\ \hline
$33$\rule{0mm}{4mm} & $35$ & $57.173821\ldots$ & $369$ & $57.318728\ldots$ & $8.579765\ldots$ & $8.585262\ldots$ \\ \hline
$31$\rule{0mm}{4mm} & $33$ & $52.857620\ldots$ & $343$ & $52.860815\ldots$ & $8.429814\ldots$ & $8.434750\ldots$ \\ \hline
$29$\rule{0mm}{4mm} & $31$ & $48.603628\ldots$ & $316$ & $48.617914\ldots$ & $8.261174\ldots$ & $8.273065\ldots$ \\ \hline
$28$\rule{0mm}{4mm} & $29$ & $44.446746\ldots$ & $292$ & $44.475817\ldots$ & $8.100022\ldots$ & $8.187500\ldots$ \\ \hline
$27$\rule{0mm}{4mm} & $28$ & $42.409320\ldots$ & $280$ & $42.489045\ldots$ & $8.017368\ldots$ & $8.098451\ldots$ \\ \hline
$26$\rule{0mm}{4mm} & $27$ & $40.379936\ldots$ & $268$ & $40.539655\ldots$ & $7.928286\ldots$ & $8.005631\ldots$ \\ \hline
$25$\rule{0mm}{4mm} & $26$ & $38.367099\ldots$ & $253$ & $38.384319\ldots$ & $7.825677\ldots$ & $7.908711\ldots$ \\ \hline
$24$\rule{0mm}{4mm} & $25$ & $36.362777\ldots$ & $239$ & $36.465241\ldots$ & $7.732514\ldots$ & $7.807320\ldots$ \\ \hline
$23$\rule{0mm}{4mm} & $24$ & $34.376006\ldots$ & $225$ & $34.391522\ldots$ & $7.629811\ldots$ & $7.701036\ldots$ \\ \hline
$22$\rule{0mm}{4mm} & $23$ & $32.426616\ldots$ & $213$ & $32.528199\ldots$ & $7.526706\ldots$ & $7.589372\ldots$ \\ \hline
$21$\rule{0mm}{4mm} & $22$ & $30.507538\ldots$ & $197$ & $30.676941\ldots$ & $7.422168\ldots$ & $7.471768\ldots$ \\ \hline
$20$\rule{0mm}{4mm} & $21$ & $28.609911\ldots$ & $182$ & $28.850866\ldots$ & $7.313019\ldots$ & $7.347570\ldots$ \\ \hline
$19$\rule{0mm}{4mm} & $20$ & $26.746588\ldots$ & $167$ & $26.764506\ldots$ & $7.180986\ldots$ & $7.216013\ldots$ \\ \hline
$18$\rule{0mm}{4mm} & $19$ & $24.895329\ldots$ & $153$ & $24.896744\ldots$ & $7.036674\ldots$ & $7.076191\ldots$ \\ \hline
$17$\rule{0mm}{4mm} & $18$ & $23.069254\ldots$ & $142$ & $23.172469\ldots$ & $6.906365\ldots$ & $6.927020\ldots$ \\ \hline
$16$\rule{0mm}{4mm} & $17$ & $21.283925\ldots$ & $130$ & $21.296251\ldots$ & $6.759540\ldots$ & $6.767193\ldots$ \\ \hline
\end{tabular}
\end{center}
Following the argumentation from the proof of Theorem \ref{thm104}, we get that the required inequality \eqref{1.13} holds for every integer $n$ with $N_{16} \leq n < N_{39}$. Next, we consider the case where $n$ satisfies $N_{15} \leq n < N_{16}$. Here, we have $17.788\ldots \leq \log_{10} n < 19.513\ldots$. By Noe \cite{noe}, we have $\log_{10} SA_{117} = 19.264\ldots$. So if $n$ fulfills $N_{15} \leq n < SA_{117}$, we can use the definition of $SA_{117}$ as a superabundant number to see that
\begin{equation}
\frac{\sigma(n)}{n} \leq \frac{\sigma(SA_{117})}{SA_{117}} = 6.589\ldots < G(15) = G(K(n)). \tag{6.2} \label{6.2}
\end{equation}
On the other hand, we need to consider the case where $n$ satisfies $SA_{117} \leq n < N_{16}$. Note that $\log_{10} N_{16} < 19.565\ldots = \log_{10} SA_{118}$ (cf. Noe \cite{noe}). Hence we obtain that $\sigma(n)/n \leq \sigma(SA_{117})/SA_{117}$ and the computation \eqref{6.2} provides the inequality \eqref{1.13}. Finally, we need to show that the inequality \eqref{1.13} also holds for every integer $n$ with $521,585,633,051,683,200 < n < N_{15}-1$. Since $N_{14} < SA_{107} < 521,585,633,051,683,200$, it suffices to apply Proposition \ref{prop601} to see that
\begin{displaymath}
\frac{\sigma(n)}{n} \leq \frac{\sigma(SA_{106})}{SA_{106}} = 6.40729\ldots G(14) = G(K(n))
\end{displaymath}
and we arrive at the end of the proof.
\end{proof}

\section{Proof of Corollary \ref{kor106}}

Finally, we present a proof of Corollary \ref{kor106}, in which we give an answer to the question raised by Aoudjit, Berkane, and Dusart \cite[Question 1]{berkane} whether the inequality \eqref{1.9} holds unconditionally for every integer $n \geq 205$.

\begin{proof}[Proof of Corollary \ref{kor106}]
According to Corollary \ref{kor105}, the required inequality holds for every integer $n \geq N_{15} > 521,585,633,051,683,200$. Next, we verify with a computer that 
\begin{displaymath}
\log_2 N_{k+1} < \log k + \log_2 k + \frac{\log_2 k}{\log k} \tag{7.1} \label{7.1}
\end{displaymath}
for every integer $k$ so that $5 \leq k < 15$. Now let $n$ be an integer with $N_{5} \leq n < N_{15}$ und let $k$ be the unique positive integer so that $N_k \leq n < N_{k+1}$. Then $5 \leq k < 15$ and $K(n) = k$, and \eqref{7.1} implies
\begin{displaymath}
e^{\gamma}\log_2 n < e^{\gamma} \left( \log K(n) + \log_2 K(n) + \frac{\log_2 K(n)}{\log K(n)} \right).
\end{displaymath}
Now we can use Lemma \ref{lem504} to see that the required inequality also holds for every integer $n$ with $N_{5} \leq n < N_{15}$. Since $N_5 = 2,310$, we can verify the required inequality for the remaining positive integers $n$ directly with a computer.
%
%
%
%
%
%
\end{proof}

\section*{Acknowledgements}
The author wishes to thank the two beautiful souls R. and O. for the never ending inspiration. The author would also like to 
thank the anonymous referee for the useful comments and suggestions to improve the quality of this paper.

\end{document}